\documentclass[a4paper]{article}

\usepackage{amsmath,amssymb,color}

\usepackage{graphicx}
\usepackage{subfigure}
\usepackage{subdepth}
\usepackage{stmaryrd} 
\usepackage{lineno,etoolbox,units}

\allowdisplaybreaks
\DeclareGraphicsExtensions{.pdf,.eps}

\newcommand{\vertiii}[1]{{\left\vert\kern-0.25ex\left\vert\kern-0.25ex\left\vert #1 
    \right\vert\kern-0.25ex\right\vert\kern-0.25ex\right\vert}}
\newcommand{\bu}{\boldsymbol u}

\newcommand{\di}{\mathrm d}

\newtheorem{Theorem}{Theorem}[section]
\newtheorem{lema}[Theorem]{Lemma}

\newtheorem{remark}[Theorem]{Remark}

\newtheorem{Proof}{{\em Proof:}}

\newenvironment{proof}{\begin{Proof}\rm}{\hfill $\Box$ \end{Proof}}

\begin{document}
\title{Pointwise error bounds in POD methods without difference quotients}
\author{Bosco
Garc\'{\i}a-Archilla\thanks{Departamento de Matem\'atica Aplicada
II, Universidad de Sevilla, Sevilla, Spain. Research is supported
by grants PID2021-123200NB-I00 and PID2022-136550NB-I00  funded by MCIN/AEI/
10.13039/501100011033 and by ERDF A way of making Europe, by
the European Union. (bosco@esi.us.es)}
  \and Julia Novo\thanks{Departamento de
Matem\'aticas, Universidad Aut\'onoma de Madrid, Spain. Research is supported
by grant PID2022-136550NB-I00  funded by MCIN/AEI/
10.13039/501100011033 and by ERDF A way of making Europe, by
the European Union. (julia.novo@uam.es)} }
\date{\today}

\maketitle
\begin{abstract} In this paper we consider proper orthogonal decomposition (POD) methods that do not include difference quotients (DQs) of snapshots  in the data set. The inclusion of DQs  have been shown in the literature to be a key element in obtaining error bounds that do not degrade with the number of snapshots. More recently, the inclusion of DQs has allowed to obtain pointwise (as opposed to averaged) error bounds that decay with the same convergence rate (in terms of the POD singular values) as averaged ones. In the present paper, for POD methods not including DQs in their data set, we obtain error bounds that do not degrade with the number of snapshots if the function from where the snapshots are taken has certain  degree of smoothness. Moreover,  the rate of convergence is as close as that of methods including DQs as the smoothness of the function providing the snapshots allows. We do this by obtaining discrete counterparts of Agmon and interpolation inequalities in Sobolev spaces. Numerical experiments validating these estimates are also presented.
\end{abstract}


\section{Introduction}

There seems to be an unclosed debate on wether it is necessary or advisable to 
include the difference quotients (DQs) of the snapshots (or function values) in the data set in proper orthogonal decomposition (POD) methods.
These allow for a remarkable dimension reduction in large-scale dynamical systems by projecting the equations onto smaller spaces spanned by the first elements of the POD basis. This basis is extracted from the data set, originally the set of snapshots,  although, as mentioned, some authors include their DQs.                
Inclusion of these has been essential to obtain optimal error bounds~\cite{Ku-Vol,iliescu_et_al_q,singler,samu_et_al} for POD methods.
More recently, pointwise (in time) error bounds have been proved in case DQs are added to the set of snapshots \cite{samu_et_al}, or if the DQs are complemented with just one snapshot~\cite{locke-singler}.
In fact,  counterexamples are presented in~\cite{samu_et_al} showing that if DQs are not included, pointwise projection errors degrade with the number of snapshots. However, this degradation is hard to find in many practical instances, 
 and, at the same time, while some authors report improvement in their numerical simulations if DQs are included in the data set~\cite{iliescu_et_al_q}, others find just the opposite~\cite{nos_pos_supg,kean_sch}.

The present paper does not pretend to settle the argument, but hopes to shed some light on why, in agreement with the numerical experience of some authors,
it may be possible to obtain good results in practice without including DQs. In particular, we show that if a function has first derivatives (with respect to time) square-integrable in time, POD projection errors do not degrade with the number of snapshots.
Furthermore, if second or higher-order derivatives are square-integrable in time, POD methods are convergent and, the higher the order of the derivatives that are square integrable, the closer to optimal the rate of convergence becomes. We do this by obtaining discrete versions of Agmon and interpolation inequalities in Sobolev
spaces, which allow to bound the $L^\infty$ norm of a function in terms of the $L^2$ norm and higher-order Sobolev's norms. 

We remark that our results do not contradict previous results in the literature. On the one hand, the error bounds we prove are as close to optimal as the smoothness of the solution allows, but not optimal.  On the other hand, concerning the counterexamples~in~\cite{samu_et_al}, we notice that they require the snapshots to be multiples of the POD basis and,
hence,  an orthogonal set, a property that is lost if a function is smooth and the snapshots are function values taken at increasingly closer times as the number of snapshots increases.


The reference \cite{Ku-Vol} is the first
one showing convergence results for POD methods for parabolic problems. Starting with this paper  one can find several references
in the literature developing the error analysis of POD methods. We mention some references that are not intended to be a complete list. 
In \cite{Kunisch-Volkwein2} the same authors of \cite{Ku-Vol} provide error estimates for POD methods for nonlinear parabolic systems 
arising in fluid dynamics. The authors of \cite{chapelle_et_al} show that comparing the POD approximation with the $L^2$ projection
over the reduced order space instead of the $H_0^1(\Omega )$-projection (as in \cite{Ku-Vol}) DQs are
not required in the set of snapshots. However, in that case,
the norm of the gradient of the $L^2$-projection has to be bounded. This leads to non-optimal estimates for the truncation errors, see \cite{samu_et_al}. POD approximation errors in different norms and using different projections can also be found in \cite{singler}.

The rest of the paper is as follows. In Section~\ref{se:main} we present the main results of this paper and their application to obtain pointwise error bounds in POD methods based only on snapshots. Section~\ref{se:numer} present some numerical experiments. In Section~\ref{se:abstract}, the auxiliary results needed to prove the main results in~Section~\ref{se:main} are stated and proved, and the final section contains the conclusions.

\section{Main results}\label{se:main}
Let $\Omega$ be a bounded region in~${\mathbb R}^d$ with a smooth boundary or a Lipschitz polygonal boundary. We use standard notation, so that $H^s(\Omega)$ denotes the standard Sobolev space of index~$s$.
We also consider the space~$H^1_0(\Omega)$ of functions vanishing on the boundary and with square integrable first derivatives.
Let $T>0$ and let $u: \Omega\times [0,T] \rightarrow {\mathbb R}$ be a function  
such that its partial derivative with respect to its second argument, denoted by~$t$ henceforth, satisfies that $u\in H^m(0,T,X)$, where $X$ denotes either~$L^2(\Omega)$ or $H^1_0(\Omega)$, and
\begin{equation}
\label{norma_Hm0TX}
\left\| u\right\|_{H^m(0,T,X)} =\biggl( \sum_{k=0}^m \frac{1}{T^{2(m-j)}}\int_0^T \bigl\| \partial_t^k u(\cdot,t)\bigr\|_X^2\,\di t\biggr)^{1/2}
\end{equation}
and~$\left\|\cdot\right\|_X$ denotes the norm in~$X$, that is $\left\| f\right\|_X=(f,f)^{1/2}$ if $X=L^2(\Omega)$ or~$\left\| f\right\|_X=(\nabla f,\nabla f)^{1/2}$ if $X=H^1_0(\Omega)$. Typically, in many practical instances, $u$ is a finite element approximation to the solution of some time-dependent partial differential equation. The factors~$1/T^{2(m-j)}$ in the definition of $\left\| \cdot\right\|_{H^m(0,T,X)}$ are introduced so that the constants in the estimates that follow are scale invariant~(see e.g., \cite[Chapter 4]{Constantin-Foias}). In what follows, if a constant is denoted with a lower case letter, this means that it is scale-invariant.

Given a positive integer $M>0$ and the corresponding time-step $\Delta t=T/M$, we consider the time levels $t_n=n\Delta t$, $n=0,\ldots,M$, and the function values or snapshots of~$u$
$$
u^n = u(\cdot,t_n),\qquad n=0,\ldots,M.
$$
We denote by $\sigma_1\ge \ldots \ge \sigma_{J}$ the singular values and by $\varphi^1,\ldots,\varphi^{J}$ the left singular vectors of the operator from ${\mathbb R}^{M+1}$ (endorsed with the standard Euclidean norm) which maps every vector $x=[x_1,\ldots,x_{M+1}]^T$ to
\begin{equation}
\label{elK}
Kx=\frac{1}{\sqrt{M+1}}\sum_{n=0}^M x_{n+1} u^n,
\end{equation}
operator which, being of finite rank, is compact.
 The singular values satisfy that $\lambda_k=\sigma_k^2$, for $k=1,\ldots,J$ are the positive eigenvalues of the correlation matrix $M=\frac{1}{M+1}((u^i,u^j)_X)_{0\le i,j\le M}$.
The left singular vectors, known as the POD basis, are mutually orthogonal and with unit norm in~$X$. For $1\le r\le d$ let us denote by
$$
V_r=\hbox{\rm span}(\varphi^1,\ldots,\varphi^r),
$$
and by~$P^r_X:X\rightarrow V_r$ the orthogonal projection of~$X$ onto~$V_r$. It is well-known that
\begin{equation}
\label{basic_id}
\frac{1}{M+1}\sum_{n=0}^M \bigl\|u^n - P_X^r u^n\bigr\|_X^2 =\sum_{k>r} \sigma_k^2.
\end{equation}
In the present paper, the rate of decay given by the square root of the right-hand side above will be referred to as {\em optimal} (see \cite{iliescu_et_al_q}).

In many instances of interest in practice, at least when $u$ is a finite-element approximation to the solution of a dissipative partial differential equation (PDE), the singular values decay exponentially fast, so that for very small values of~$r$ (compared with the dimension of the finite-element space) the right-hand side of~\eqref{basic_id} is smaller than the error of~$u$ with respect the solution of the PDE it approximates. 

POD methods are Galerkin methods where the approximating spaces are the spaces~$V_r$. In POD methods, one obtains values~$u_r^n\in V_r$ intended to be approximations~$u_r^n\approx u^n$, $n=0,\ldots, M$, and~\eqref{basic_id} is key in obtaining estimates for~the errors $u_r^n-u^n$. Yet, the fact that the left hand side~\eqref{basic_id} is an average allows only for averaged estimates of the errors $u_r^n-u^n$, and one would be interested
in pointwise estimates. In order to obtain them, some authors assume (without a result to support it) that the individual errors $\bigl\| u^n-P_X^r u^n\bigr\|_X$
decay with~$r$ as the square root of the right-hand side of~\eqref{basic_id} (see e.g., \cite[(2.9)]{iliescu_wang}).
%

{Yet, in principle, the only rigorous {\em pointwise} estimate that one can deduce from~\eqref{basic_id} is
\begin{equation}
\label{degraded_pointwise}
 \bigl\|u^n - P_X^r u^n\bigr\|_X \le \sqrt{M+1} \biggl(\sum_{k>r} \sigma_k^2\biggr)^{\frac{1}{2}}, \qquad n=0,\ldots,M.
\end{equation}
This estimate is sharp as shown in examples provided in~\cite{samu_et_al}, where for at least one value of~$n$, the equality is reached. Estimate~\eqref{degraded_pointwise} degrades with the number $M+1$ of snapshots. One would expect then in practice that if $M$ is increased while maintaining the value of~$r$ the individual errors
$ \bigl\|u^n - P_X^r u^n\bigr\|_X$ should increase accordingly. Yet this is not necessary the case as results in Table~\ref{tab_T1r25} show, with correspond to the snapshots of the quadratic FE approximation to the periodic solution of the Brusselator problem described in~\cite{nos_semi} (see also Section~\ref{se:numer} below) in the case where~$X=L^2(\Omega)$, $r=25$ and~$T$ is one period.
\begin{table}
\begin{center}
$$
\begin{array}{|c|c|c|c|c|}
\hline
M& 128 & 256& 512 & 1024\\
\hline
\hbox{\rm\textcolor{black}{max.\ err.}} &1.12\times 10^{-4} &    1.12\times 10^{-4} &     1.14\times 10^{-5} &     1.15\times 10^{-5}\\
\hline
\end{array}
$$
\end{center}
\caption{\label{tab_T1r25} Maximum pointwise projection error  $\max_{0\le n\le M}\bigl\|u^n - P_X^r u^n\bigr\|_X$ for different values of~the number of snapshots $M+1$
and $u$ being
the quadratic FE approximation to the periodic solution of the Brusselator problem described in~\cite{nos_semi}, when~$X=L^2(\Omega)$,~$r=25$ and $T$ is one period.}
\end{table}
The results in Table~\ref{tab_T1r25} are in agreement with the following result.
\begin{Theorem}\label{th:1} Let $u$ be bounded in~$H^1(0,T,X)$. Then, for the constant~$c_A$ defined in~\eqref{c_A} below, the following bound 
holds:
$$
\max_{0\le n\le M}\bigl\|u^n - P_X^r u^n\bigr\|_X \le c_A \biggl(\sum_{k>r} \sigma_k^2\biggr)^{\frac{1}{4}}\bigl\| (I-P_X^r)\partial_t u\bigr)\|_{L^2(0,T,X)}^{\frac{1}{2}} + \frac{1}{\sqrt{T}}
\biggl(\sum_{k>r} \sigma_k^2\biggr)^{\frac{1}{2}}.
$$
If $u_0+\cdots+u_M=0$, then the second term on the right-hand side above can be omitted.
\end{Theorem}
\begin{proof} The proof is a direct consequence of the case $m=1$ in estimate~\eqref{cota_Df_g} in Theorem~\ref{th:cotas_f_g} below when applied to~$f=u-P_X^r u$ and 
of identity~\eqref{basic_id} above. Then norms~$\left\|cdot\right\|_0$ in~\eqref{cota_Df_g} are defined in~(\ref{norma0}--\ref{norma1}).
\end{proof}

Observe that in the pointwise estimate in Theorem~\ref{th:1} no factor depending on $M$ appears on the right-hand side, except, perhaps, the singular values, but if~$u\in H^1(0,T,X)$ then, $\sigma_k\rightarrow \sigma_{c,k}$ as~$M\rightarrow \infty$, where $\sigma_{c,1}\ge \sigma_{c,2} \ge \ldots\,$, are the singular values of the operator $K_c:L^2(0,T)\rightarrow X$ given by \cite[Section~3.2]{Kunisch-Volkwein2}
$$
K_cf=\frac{1}{T} \int_0^t f(t)u(t)\,\di t.
$$
Observe also that, at worst,
$$
\bigl\| (I-P_X^r)\partial_t u\bigr\|_{L^2(0,T,X)} \le \bigl\| \partial_t u\|_{L^2(0,T,X)}.
$$

Thus, Theorem~\ref{th:1} guarantees that pointwise projection errors do not degrade as the number of snapshots increases. Furthermore, they decay with $r$, although like
$\left(\Sigma_{k>r} \sigma_k^2\right)^{1/4}$,
which is the square root of the {\rm optimal} rate. Yet, the following result shows that pointwise projection errors decay with a rate as close to optimal as the smoothness of~$u$ allows.


\begin{Theorem}\label{th:2}
Let $u$ be bounded in~$H^m(0,T,X)$. Then, for the constants~$c_A$ defined in~~\eqref{c_A} and~$c_m$ in~Lemma~\ref{le:esti_maxg}
 below, 
the following bound
holds:
\begin{align*}
\max_{0\le n\le M} 
\bigl\|u^n - P_X^r u^n\bigr\|_X \le& \sqrt{2}  c_Ac_m^{\frac{1}{2}} \biggl(\sum_{k>r} \sigma_k^2\biggr)^{\frac{1}{2}-\frac{1}{4m}}\bigl\| (I-P_X^r)\partial_t u\bigr\|_{H^{m-1}(0,T,X)}^{\frac{1}{2m}}
\nonumber
\\
&{}+ \frac{1}{\sqrt{T}}\biggl(\sum_{k>r} \sigma_k^2\biggr)^{\frac{1}{2}}.
\end{align*}
If $u_0+\cdots+u_M=0$, then the second term on the right-hand side above can be omitted.
\end{Theorem}
\begin{proof} The proof is a direct 
consequence of~estimate~\eqref{cota_Df_g} in Theorem~\ref{th:cotas_f_g} below 
applied to~$f=u-P_X^r u$ and 
of identity~\eqref{basic_id} above.
\end{proof}


The above results account for projection errors, but in the analysis of POD methods the error between the POD approximation and the snapshots, $u_r^n - u^n$ depend also on
the difference quotients
\begin{equation}
\label{DQ}
Du^n=\frac{u^n-u^{n-1}}{\Delta t}, \qquad n=1,\ldots,M.
\end{equation}
for wich we have the following result.

\begin{Theorem}\label{th:3} In the conditions of Theorem~\ref{th:1} the following bound holds:
$$
\biggr(\Delta t\sum_{n=1}^M \bigl (I-P_X^r)Du^n\bigr\|_X^2\biggr)^{\frac{1}{2}}\le 
2c_m \biggl(\sum_{k>r} \sigma_k^2\biggr)^{\frac{m-1}{2m}} \bigl\| (I-P_X^r)\partial_t u\bigr)\|_{H^{m-1}(0,T,X)}^{\frac{1}{m}}.
$$
\end{Theorem}
\begin{proof} The proof is a direct consequence of estimate~\eqref{cota_Df_g_D1} in Theorem~\ref{th:cotas_f_g} below 
applied to~$f=u-P_X^r u$ and 
of identity~\eqref{basic_id} above.
\end{proof}

\begin{remark}\label{re:periodic}\rm An important case in practice in that of periodic orbits, since they are key elements in bifurcation diagrams of many dynamical systems associated to PDEs. In this case, the snapshot $u_0$ can be omitted in \eqref{elK}  and~\eqref{basic_id}, and $M+1$ can be replaced by~$M$ in previous formulae (see also Section~\ref{se:periodic}). Also, in the estimates in Theorems~\ref{th:1} and~\ref{th:2} one can replace the constant~$c_m$ by~$1$, and $\bigl\| (I-P_X^r)\partial_t u\bigr\|_{H^{m-1}(0,T,X)}$ by
$\bigl\| (I-P_X^r)\partial_t^m u\bigr\|_{L^2(0,T,X)}$, that is
\begin{align}
\label{esti1}
\max_{0\le n\le M} \bigl\|u^n - P_X^r u^n\bigr\|_X \le& \sqrt{2}  c_A\biggl(\sum_{k>r} \sigma_k^2\biggr)^{\frac{1}{2}-\frac{1}{4m}}\bigl\| (I-P_X^r)\partial^m_t u\bigr\|_{L^2(0,T,X)}^{\frac{1}{2m}}
\nonumber
\\
&{}+ \frac{1}{\sqrt{T}}\biggl(\sum_{k>r} \sigma_k^2\biggr)^{\frac{1}{2}},
\end{align}
\begin{equation}
\label{esti2}
\biggr(\Delta t\sum_{n=1}^M \bigl\| (I-P_X^r)Du^n\bigr\|_X^2\biggr)^{\frac{1}{2}}\le 
2\biggl(\sum_{k>r} \sigma_k^2\biggr)^{\frac{1}{2}-\frac{1}{2m}} \bigl\| (I-P_X^r)\partial_t^m u\bigr\|_{L^2(0,T,X)}^{\frac{1}{m}}.
\end{equation}
As in~Theorem~\ref{th:1}, the second term on the right-hand side of~\eqref{esti2} can be omitted if the snapshots have zero mean. The proof of these estimates follows by
applying~Theorem~\ref{th:cotas_f} to~$(I-P_X^r) u$.

Some reduction in the size of the constants $c_m$ can be obtained also when computing quasi periodic orbits in an invariant tori if one uses the techniques described in~\cite{Joan_toros} to compute them, since although $u(T)\ne u(0)$, one still has $\left\| u(T)-u(0)\right\|\ll 1$, and one can modify the technique so that this is also the case with first derivatives or higher order ones.
\end{remark}

\begin{remark}\label{re:diferente}\rm In the analysis of POD methods, as it will be the case in next section, one usually has to estimate~$(I-P_X^r)$ in a norm $\|\cdot\|$ other than that of~$X$. If $\left\|\cdot \right\|$ is the norm of a Hilbert space~$W$ containing the snapshots~$u^0,\ldots, u^M$, then one can use \cite[Lemma~2.2]{samu_et_al} to get
\begin{equation}
\label{different_norms}
\frac{1}{M+1} \sum_{n=0}^m \bigl\| (I-P_X^r)u^n\bigr\|^2 = \sum_{k>r} \sigma_k^2 \bigl\| \varphi^k\bigr\|^2
\end{equation}
Applying this estimate instead of the identity~\eqref{basic_id} in the proofs of Theorems~\ref{th:2} and~\ref{th:3} one gets the following estimates
\begin{align}
\label{diferente1}
\max_{0\le n\le M} \bigl\|u^n - P_X^r u^n\bigr\| \le&\sqrt{2}  c_Ac_m^{\frac{1}{2}} \biggl(\sum_{k>r} \sigma_k^2 \bigl\| \varphi^k\bigr\|^2\biggr)^{\frac{1}{2}-\frac{1}{4m}}\bigl\| (I-P_X^r)\partial_t u\bigr\|_{H^{m-1}(0,T,W)}^{\frac{1}{2m}}
\nonumber
\\
&{}+ \frac{1}{\sqrt{T}}\biggl(\sum_{k>r} \sigma_k^2\bigl\| \varphi^k\bigr\|^2\biggr)^{\frac{1}{2}}.
\end{align}
\begin{equation}
\label{diferente2}
\biggr(\Delta t\sum_{n=1}^M \bigl (I-P_X^r)Du^n\bigr\|^2\biggr)^{\frac{1}{2}}\!\!\le 
2c_m \biggl(\sum_{k>r} \sigma_k^2 \bigl\| \varphi^k\bigr\|^2\biggr)^{\frac{m-1}{2m}}\! \bigl\| (I-P_X^r)\partial_t u\bigr\|_{H^{m-1}(0,T,W)}^{\frac{1}{m}}.
\end{equation}
As in~Remark~\ref{re:periodic}, when~$u$ is periodic in~$t$ with period~$T$, then $c_m$ can be replaced by~$1$ and~$\bigl\| (I-P_X^r)\partial_t u\bigr\|_{H^{m-1}(0,T,W)}$ by~$\bigl\| (I-P_X^r)\partial_t^m u\bigr)\|_{L^2(0,T,W)}$.
\end{remark}

\subsection{Application to a POD method.}\label{se:application}
We now apply the previous results to obtain pointwise error bounds for a POD method applied to the heat equation where no DQ were included in the data set to obtain the POD basis $\{\varphi_1,\ldots,\varphi_J\}$. To discretize the time variable we consider two methods, the implicit Euler method and the second-order backward differentiation formula (BDF2). Analysis similar to part of computations done in the present section has been done before (see e.g., \cite[Lemma~3]{nos_ap_let}) but it is carried out here because better error constants are obtained.
%

Let us assume that $u$ is a semi-discrete FE approximation to the solution of the heat equation with a forcing term~$f$ so that $u(t)\in V$ for some FE space~$V$ of piecewise polynomials over a triangulation and satisfies
\begin{equation}
\label{heat_FEM}
(\partial_t u ,\varphi) + \nu(\nabla u,\nabla \varphi) = (f,\varphi), \quad \forall \varphi \in V,
\end{equation}
where~$\nu>0$ is the thermal diffusivity and, in this section, $(\cdot,\cdot)$ denotes the standard inner product of~$L^2(\Omega)$.
We consider the~POD method
\begin{equation}
\label{heat_POD}
(Du_r^n,\varphi) + \nu (\nabla u_r^n ,\nabla\varphi) = (f,\varphi),\qquad\forall \varphi\in V_r, \qquad n=1,\ldots, M,
\end{equation}
with
\begin{equation}
\label{heat_condi}
u_r^0=R_ru(\cdot,0),
\end{equation}
where $R_r$ denotes the Ritz projection,
$$
(\nabla R_r u,\nabla \varphi) = (\nabla u,\nabla \varphi),\qquad \forall \varphi\in V_r.
$$
Instead of using the first-order convergent Euler method considered above we may use the BDF2, with gives the
the following method
\begin{equation}
\label{heat_POD2}
({\cal D}u_r^n,\varphi) + \nu (\nabla u_r^n ,\nabla\varphi) = (f,\varphi),\qquad\forall \varphi\in V_r, \qquad n=2,\ldots, M,
\end{equation}
where
\begin{equation}
\label{BDF2}
{\cal D} u_r^n = \frac{3}{2} D u_r^n - \frac{1}{2}Du_{n-1}, \qquad n=2,\ldots,M,
\end{equation}
and, for simplicity,
\begin{equation}
\label{BDF2_assumption}
u_r^n=P_X^r u^n,\qquad n=0,1.
\end{equation}
We now prove the following result.
\begin{Theorem}\label{th:heat} Let $X$ be $H^1_0$ and let $p=1$ in the case of method~(\ref{heat_POD}-\ref{heat_condi}) and $p=2$ otherwise. Assume that~$u\in H^{p+1}(0,T,L^2)\cap H^{m}(0,T,H^1_0)$ for some $m\ge 2$. Then, the following bound holds for $0\le n\le M$:
\begin{align}
\label{result}
\bigl\| u_r^n - u^n\bigr\| &\le 4C_P\sqrt{T}c_m  \biggl(\sum_{k>r} \sigma_k^2\biggr)^{\frac{1}{2}-\frac{1}{2m}}
  \bigl\| (I-P_X^r)\partial_t u\bigr\|_{H^{m-1}(0,T,H^1_0)}^{\frac{1}{m}}
  \nonumber\\
&{}+\sqrt{2}C_P c_Ac_m^{1/2} \biggl(\sum_{k>r} \sigma_k^2\biggr)^{\frac{1}{2}-\frac{1}{4m}}\bigl\| (I-P_X^r)\partial_t u\bigr\|_{H^{m-1}(0,T,H^1_0)}^{\frac{1}{2m}}
\nonumber
\\
&{}+ \frac{1}{\sqrt{T}}\biggl(\sum_{k>r} \sigma_k^2\biggr)^{\frac{1}{2}}
 + \sqrt{p+3}(\Delta t)^p\sqrt{T} \left\| \partial_t^{p+1}u\right\|_{L^2(0,T,L^2)}.
\end{align}
\end{Theorem}
\begin{proof}
Since the the POD basis~$\{\varphi_1,\ldots,\varphi_J\}$ has been computed with respect to the inner product in~$H^1_0$, we have $P_X^r=R_r$. We first analyze method~(\ref{heat_POD}--\ref{heat_condi}). We notice that
\begin{equation}
\label{heat_repre}
(DP_X^ru^n,\varphi) + \nu(\nabla P_X^r u^n,\nabla\varphi)=(f,\varphi) + (\tau^n,\varphi), \qquad \forall\varphi\in V_r,
\end{equation}
where
$$
\tau^n = DP_X^r u^n - \partial_tu(\cdot,t_n),\qquad n=1,\ldots,M.
$$
Subtracting \eqref{heat_repre} from~\eqref{heat_POD}, for the error
$$
e_r^n=u_r^n - P_X^r u^n,\qquad n=0,\ldots,M,
$$
we have the following relation,
\begin{equation}
\label{heat_error}
(De_r^n,\varphi) + \nu (\nabla e_r^n ,\nabla\varphi) = (\tau^n,\varphi),\qquad\forall \varphi\in V_r, \qquad n=1,\ldots, M.
\end{equation}
Taking $\varphi=\Delta t e_r^n$ and using that $(e_r^n - e_r^{n-1},e_n) = (\bigl\| e_r^n \bigr\|^2 - \bigl\| e_r^{n-1} \bigr\|^2 + \bigl\| e_r^n-e_r^{n-1} \bigr\|^2)/2$, one obtains
\begin{equation}
\label{evol_sq}
\bigl\| e_r^n \bigr\|^2 - \bigl\| e_r^{n-1} \bigr\|^2 +2\nu\bigl\| \nabla e_r^n\bigr\|^2  \le 2\bigl\|\tau^n\bigr\| \bigl\|e_r^n\bigr\|,
\end{equation}
If $\bigl\| e_r^n\bigr\|> \bigl\| e_r^{n-1}\bigr\|$, then 
$-1< -\bigl\| e_r^{n-1}\bigr\|/\bigl\| e_r^n\bigr\|$. Hence, dividing both sides of~\eqref{evol_sq} by~$\bigl\| e_r^n\bigr\|$ we get
\begin{equation}
\label{evol_nosq}
\bigl\| e_r^n \bigr\| - \bigl\| e_r^{n-1} \bigr\|  \le 2\bigl\|\tau^n\bigr\|.
\end{equation}
 If, on the contrary, $\bigl\| e_r^n\bigr\|\le \bigl\| e_r^{n-1}\bigr\|$, for the right-hand side in~\eqref{evol_sq} we write $\bigl\|\tau^n\bigr\| \bigl\|e_r^n\bigr\|\le \bigl\|\tau^n\bigr\|( \bigl\|e_r^n\bigr\| +
\bigl\|e_r^{n-1}\bigr\|)/2$
so that dividing both sides of~\eqref{evol_sq} by~$\bigl\| e_r^n\bigr\| +
\bigl\|e_r^{n-1}\bigr\|$, we also obtain~\eqref{evol_nosq}.
Summing from~$n=1$ to~$n=m$ and using~\eqref{heat_condi}
we obtain
\begin{equation}
\label{evol_sum}
\bigl\| e_r^m \bigr\| \le 2\Delta t\sum_{n=1}^m \bigl\| \tau^n\bigr\| \le 2\sqrt{T} \biggl(\Delta t\sum_{n=1}^m \bigl\| \tau^n\bigr\|^2\biggr)^{\frac{1}{2}},
\end{equation}
where, in the last step we have applied H\"older's inequality.

With respect to~$\tau^n$, by adding $\pm Du^n$ we have
$$
\bigl\| \tau_n\bigr\| \le \bigl\| (I-P_X^r) Du^n\bigr\| + \big\| Du^n- \partial_t u(\cdot,t_n)\bigr\|.
$$
If $u\in H^2(0,T,L^2)$ Taylor expansion with integral reminder allow us to write
\begin{align*}
\big\| Du^n- \partial_t u(\cdot,t_n)\bigr\| &\le \frac{1}{\Delta t}\biggl\| \int_{t_{n-1}}^{t_n} \!\!(t-t_{n-1}) \partial^2_{t} u(\cdot,t)\,\di t\biggr\|
 \le 
\sqrt{\Delta t}
\bigl\| \partial^2_t u\bigr\|_{L^2(t_{n-1},t_n,L^2)}.
\end{align*}
Thus, from~\eqref{evol_sum} we get
$$
\bigl\| e_r^m \bigr\|\le 2\sqrt{T} \biggl(\Delta t\sum_{n=1}^m\bigl\| (I-P_X^r)D u^n\bigr\| ^2\biggr)^{\frac{1}{2}} + 2\Delta t\sqrt{T} \bigl\| \partial_t^2u\bigr\|_{L^2(0,T,L^2)}.
$$
Using Poincar\'e inequality
\begin{equation}
\label{Poincare}
\bigl\| v\bigr\| \le C_P \bigl\| \nabla v\bigr\|,\qquad \forall v\in H^1_0(\Omega),
\end{equation}
and applying~Theorem~\ref{th:3} we have
\begin{align*}
\max_{0\le n \le M} \bigl\| e_r^n\bigr\| \le& 4c_mC_P\sqrt{T}
 \biggl(\sum_{k>r} \sigma_k^2\biggr)^{\frac{m-1}{2m}}
  \bigl\| (I-P_X^r)\partial_t u\bigr\|_{H^{m-1}(0,T,H^1_0)}^{\frac{1}{m}}
  \nonumber\\
&{} + 2\Delta t\sqrt{T} \left\| \partial_t^2u\right\|_{L^2(0,T,L^2)}.
\end{align*}
Finally, to estimate the errors $u^n-u_r^n$ we write $u_r^n - u^n = e_r^n + P_X^r u^n-u^n$, and apply~\eqref{Poincare} and~Theorem~\ref{th:2}
to conclude~\eqref{result}.

The analysis of method~(\ref{heat_POD2}--\ref{BDF2_assumption}) is very similar to that of~\eqref{heat_POD}, and we comment on the differences next.
As it is well-known (and a simple calculation shows) $({\cal D}e_r^n,e_r^n)={\cal E}_n^2 - {\cal E}_{n-1}^2$, where
$$
{\cal E}_n =\frac{1}{2} \left( \bigl\| e_r^n\bigr\|^2 + \bigl\| 2e_r^n - e_r^{n-1}\bigr\|^2\right)^{1/2}, \quad n=1,\ldots,M.
$$
It is easy to check that $\left\| e_r^n \right\| \le {\cal E}_n$, so that one can repeat (with obvious changes) the analysis above with the implicit Euler method with $\left\| e_r^n\right\|$ replaced by ${\cal E}_n$ to reach
\begin{equation}
\label{evol_sum2}
\bigl\| e_r^m \bigr\| \le {\cal E}_m \le \Delta t\sum_{n=2}^m \bigl\| {\cal T}^n\bigr\| \le \sqrt{T} \biggl(\Delta t\sum_{n=2}^m \bigl\| {\cal T}^n\bigr\|^2\biggr)^{\frac{1}{2}},
\end{equation}
instead of~\eqref{evol_sum}, where  we have used that ${\cal E}_1=0$ due to~\eqref{BDF2_assumption},
and where
$$
{\cal T}^n = P_X^r {\cal D} u^n -  \partial_t u(\cdot,t_n) = (P_X^r-I) {\cal D} u^n + {\cal D} u^n - \partial_t u(\cdot, t_n).
$$
For the last term above we have
(see e.g., \cite[Lemma~4]{nos_ap_let})
$$
\bigl\| {\cal D} u^n - \partial_t u(\cdot, t_n)\bigr\| \le \sqrt{5}(\Delta t)^{3/2} \biggl(\int_{t_{n-2}}^{t_n}
\bigl\| \partial_{t}^3 u(\cdot,t)\bigr\|^2\,\di t\biggr)^{\frac{1}{2}}.
$$
Consequently, from~\eqref{evol_sum2} and noticing that $\bigl\| {\cal D}e_r^n\bigr\| \le (3/2)\bigl\| De_r^n\bigr\| + (1/2)\bigl\| De_r^{n-1}\bigr\|$, we have
\begin{equation}
\label{usada}
\bigl\| e_r^m\bigr\| \le 2\sqrt{T} \biggl(\Delta t\sum_{n=1}^m \bigl\| (I-P_X^r) Du^n\bigr\|^2\biggr)^{1/2} + \sqrt{5T} (\Delta t)^2 \bigl\| \partial_t^3 u\bigr\|_{L^2(0,T,L^2)},
\end{equation}
so that applying~\eqref{Poincare} and~Theorem~\ref{th:3} it follows that
\begin{align*}
\max_{0\le n \le M} \bigl\| e_r^n\bigr\| \le& 4c_mC_P\sqrt{T}
 \biggl(\sum_{k>r} \sigma_k^2\biggr)^{\frac{m-1}{2m}}
  \bigl\| (I-P_X^r)\partial_t u\bigr\|_{H^{m-1}(0,T,H^1_0)}^{\frac{1}{m}}
  \nonumber\\
&{} + (\Delta t)^2\sqrt{5T} \left\| \partial_t^3u\right\|_{L^2(0,T,L^2)},
\end{align*}
and, for the error~$u_r^n - u = e_r^n + (I-P_X^r) u_n$, applying~Theorem~1 we get~\eqref{result}.
\end{proof}


\begin{remark}\label{re:diferente_POD}\rm In view of Remark~\ref{re:diferente}, in the above estimates we may replace the Poincar\'e constant $C_P$ by~$1$,
$\sigma_k^2$ by~$\sigma_k^2\bigl\| \varphi_k\bigr\|^2$ and the norm $\left\|\cdot\right\|_{H^{m-1}(0,T,H^1_0)}$ by~$\left\|\cdot\right\|_{H^{m-1}(0,T,L^2)}$.
Also in view of~Remark~\ref{re:periodic}, when $u$ is periodic in~$t$ with period~$T$, the constant $c_m$ can be replaced by~$1$ and
$\bigl\| (I-P_X^r)\partial_t u\bigr\|_{H^{m-1}(0,T,H^1_0)}$ by $\bigl\| (I-P_X^r)\partial_t^m u\bigr\|_{L^2(0,T,H^1_0)}$ (or
by~$\bigl\| (I-P_X^r)\partial_t^m u\bigr\|_{L^2(0,T,L^2)}$ if we use the
estimates in~Remark~\ref{re:diferente}).
\end{remark}

\begin{remark} \label{re:Euler1}\rm  If, for the method~\eqref{heat_POD2}, $u_r^1$ is computed from $u_r^0$ with one step of the implicit Euler method, instead of setting~$u_r^1=P_X^r u^1$, then one can use a general stability result for the BDF2 like~\cite[Lemma~3]{nos_ap_let} and then, use standard estimates
to express the term $Du^1 - \partial_t u(\cdot, t_1)$ in~$\tau^1$ as
$
\big\| Du^1- \partial_t u(\cdot,t_n)\bigr\| \le 
{\Delta t}
\bigl\| \partial^2_t u\bigr\|_{L^\infty(0,t_1,L^2)}.
$
Thus, one gets estimates similar to~\eqref{result} but with different constants multiplying the different terms and $\bigl\| \partial_t^3u\bigr\|_{L^2(0,T,L^2)}$ replaced
by $\bigl\| \partial_t^3u\bigr\|_{L^2(0,T,L^2)} + \bigl\| \partial^2_t u\bigr\|_{L^\infty(0,t_1,L^2)}$.
\end{remark}

\section{Numerical experiments}
\label{se:numer}
We now check some of the estimates of previous sections. For the snapshots, we consider the semi-discrete piecewise quadratic FE approximation over a uniform $80\times 80$ triangulation (with diagonals running southwest-northeast) of the unit square of the stable periodic orbit of the following system:
\begin{equation}
\label{bruss}
\begin{array}{rclcl}
u_t&=&\nu\Delta u +1 +u^2v  - 4 u,&\qquad& (t,x)\in \Omega \times (0,T],\\
v_t &=&\nu \Delta v+3 u -u^2v,& \qquad
& (t,x)\in \Omega \times (0,T],\\
&&u(x,t)=1, \quad v(x,t) =3,&& (t,x)\in \Gamma_1\times(0,T],\\
&&\partial_n u(x,t)=\partial_n v(t,x) =0,&&  (x,t)\in \Gamma_2\times(0,T],\\
\end{array}
\end{equation}
where $\nu=0.002$, $\Omega=[0,1]\times[0,1]$, $\Gamma_1\subset\partial\Omega$ is the union of of sides
$\{x=1\}\bigcup \{ y=1\}$, and~$\Gamma_2$ is the rest of~$\partial\Omega$ (see \cite{nos_semi} for details). The period is~$T=7.090636$ (up to 7 significant digits).
This system is particularly challenging for the size of its time derivatives, which, as it can be seen in Fig.~\ref{fig:derivs}, they vary considerably over one period.
\begin{figure}[h]
\begin{center}
\includegraphics[height=2.3truecm]{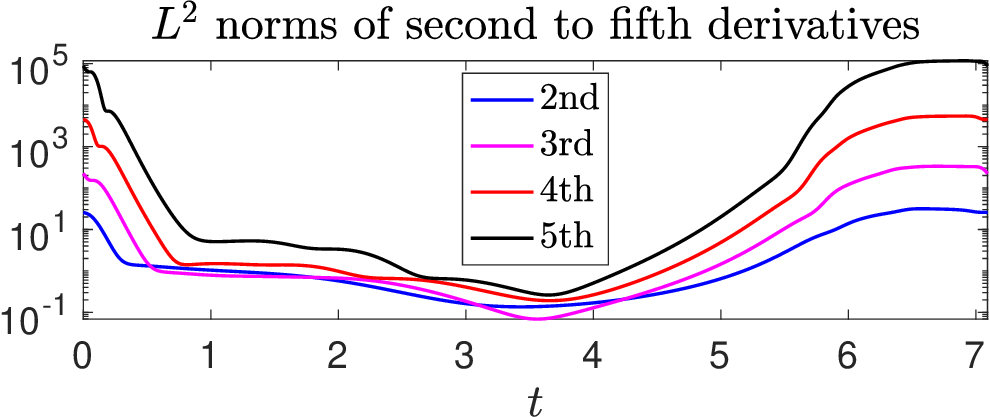}
\begin{caption}{\label{fig:derivs}  $L^2$ norms of the time derivatives of the finite element approximation $\bu_h$ to the periodic orbit solution of~\eqref{bruss}.}
\end{caption}
\end{center}
\end{figure}

In what follows, $M=256$. We took values $\bu(\cdot, t_n)=[u(\cdot,t_n),v(\cdot, t_n)]^T$, $n=0,\ldots,M$, and, in order to have~snapshots taking value zero on~$\Gamma_1$ we subtracted their mean so that $\bu^n = \bu(\cdot,t_n) - m_0$, where $m_0=(\bu(\cdot,t_1)+\cdots+\bu(\cdot,t_M))/M$ (notice that being $\bu$ periodic with period~$T$, there is no need to include~$\bu(\cdot,0)$ in the data set to obtain the POD basis, since $\bu(\cdot,0)=\bu(\cdot,T)$). 

We first check estimate~\eqref{esti1} for $X=L^2$ and two values of~$r$, $r=25$ and~$r=32$. The value of $r=25$ was chosen because it was the first value of~$r$ for which $\gamma_r$ in~\eqref{values1_sigmas} was below~$5.35\times 10^{-5}$, which, as explained in~\cite{nos_semi}, is the maximum error of the snapshots with respect the true solution of~\eqref{bruss} (recall that the snapshots are a FE approximations, not the true solution). The value~$r=32$ was chosen for comparison.
 We have the following values for the left-hand side and tail of the singular values
\begin{equation}
\label{values1_max}
\max_{1\le n\le M} \bigl\| (I-P_X^r)\bu^n\bigr\|_X =\left\{\begin{array}{rcr} 1.14\times 10^{-4},&\ &r=25,\\ 1.00\times10^{-5},&& r=32,\end{array},\right.
\end{equation}
\begin{equation}
\label{values1_sigmas}
\gamma_r\equiv \biggl(\sum_{k>r} \sigma_k^2\biggr)^{1/2} =
\left\{\begin{array}{rcr} 4.40\times 10^{-5},&\ &r=25,\\ 6.42\times10^{-6},&& r=32,\end{array}\right.
\end{equation}
and Table~\ref{table:1} shows the right-hand side of~\eqref{esti1}. We notice that there is a slight overestimation by a factor of less than~$9$.
\begin{table}[h]
\begin{center}
$$
\begin{array}{|c|c|c|c|c|c|}
\hline
m & 2& 3& 4& 5\\ \hline
r=25&     7.42\times 10^{-4}  &  7.42\times 10^{-4}  &     7.41\times 10^{-4}  &    7.43\times 10^{-4} \\ \hline
r=32 &    8.30\times 10^{-5}  &  8.37\times 10^{-5}  &    8.37\times 10^{-5}  &    8.68\times 10^{-5}\\ \hline
\end{array}
$$
\caption{\label{table:1} Right-hand side of
\eqref{esti1} for~$X=L^2$ and different values of~$m$.}
\end{center}
\end{table}

%

We now check estimates~\eqref{esti2} and~\eqref{diferente2} when $W=L^2(\Omega)$ and~$\left\|\cdot\right\|$ its norm, which have been used in~\eqref{usada} in the previous section to obtain estimate~\eqref{result} and in~Remark.~\ref{re:diferente_POD}.
We have the following values
\begin{equation}
\label{values2}
\biggr(\Delta t\sum_{n=1}^M \bigl\| (I-P_X^r)D\bu^n\bigr\|^2\biggr)^{\frac{1}{2}} =\left\{
\begin{array}{lcl} 9.28\times 10^{-3},& \ &r=25\\ 8.53\times 10^{-4},&& r=32.\end{array}\right.
\end{equation}
Notice that $X=H^1_0$ but we are measuring~$ (I-P_X^r)D\bu^n$ in the $L^2$ norm. 
Applying
Poincar\'e inequality~\eqref{Poincare} and~\eqref{esti2} the quantity above can be estimated by
\begin{equation}
\label{rho_m}
\rho_m \equiv 2C_P\biggl(\sum_{k>r} \sigma_k^2\biggr)^{\frac{1}{2}-\frac{1}{2m}} \bigl\| \nabla (I-P_X^r)\partial_t^m \bu\bigr\|_{L^2(0,T,L^2)}^{\frac{1}{m}}.
\end{equation}
On the other hand according~to~Remark~\ref{re:diferente_POD}, and taking into account that we are in the periodic case, it can also be estimated by
\begin{equation}
\label{mu_m}
\mu_m \equiv 2\biggl(\sum_{k>r} \sigma_k^2\bigl\|\varphi_k\bigr\|^2\biggr)^{\frac{1}{2}-\frac{1}{2m}} \bigl\| (I-P_X^r)\partial_t^m \bu\bigr\|_{L^2(0,T,L^2)}^{\frac{1}{m}}.
\end{equation}
Table~\ref{table:3} shows the corresponding values. Notice that $C_P=\lambda^{-1/2}$, where $\lambda$ is the smallest eigenvalue of the $-\Delta$ operator subject to
homogeneous Dirichlet boundary condition on~$\Gamma_1$ and homogeneous Neumann boundary condition on~$\Gamma_2$, and, thus, $C_P=\sqrt{2}/\pi$. By comparing the values in~Table~\ref{table:3} with those in~\eqref{values2}, we can see that while $\rho_m$ is not a good estimate, $\mu_m$ slightly overestimates the quantities in~\eqref{values2} by a factor that does not reach 8.5.
\begin{table}[h]
\begin{center}
$$
\begin{array}{|c|c|c|c|c|c|}
\hline
m & 2& 3& 4& 5\\ \hline
\rho_m (r=25)&     2.78\times 10^{-1}  &  3.42\times 10^{-1}  &     3.99\times 10^{-1}  &    5.44\times 10^{-1} \\ \hline
\rho_m (r=32)&     4.23\times 10^{-2}  &  5.29\times 10^{-2}  &     6.25\times 10^{-2}  &    1.04\times 10^{-1} \\ \hline
\mu_m (r=25)&     3.72\times 10^{-2}  &  5.23\times 10^{-2}  &     6.54\times 10^{-2}  &    7.72\times 10^{-2} \\ \hline
\mu_m (r=32)&     3.08\times 10^{-3}  &  4.09\times 10^{-3}  &     4.95\times 10^{-3}  &    5.76\times 10^{-3} \\ \hline
\end{array}
$$
\caption{\label{table:3} Estimates \eqref{rho_m} and~\eqref{mu_m} for~\eqref{values2}.}
\end{center}
\end{table}
The fact that $\mu_m$ is a much better estimate than~$\rho_m$ is explained in part by the size of the values $\bigl\|\varphi_k\bigr\|$ which, as shown in~Fig.~\ref{fig:normas}, they are small and well below the constant~$C_P$ (marked with a 
red 
dotted line in the plot).
\begin{figure}[h]
\begin{center}
\includegraphics[height=2.3truecm]{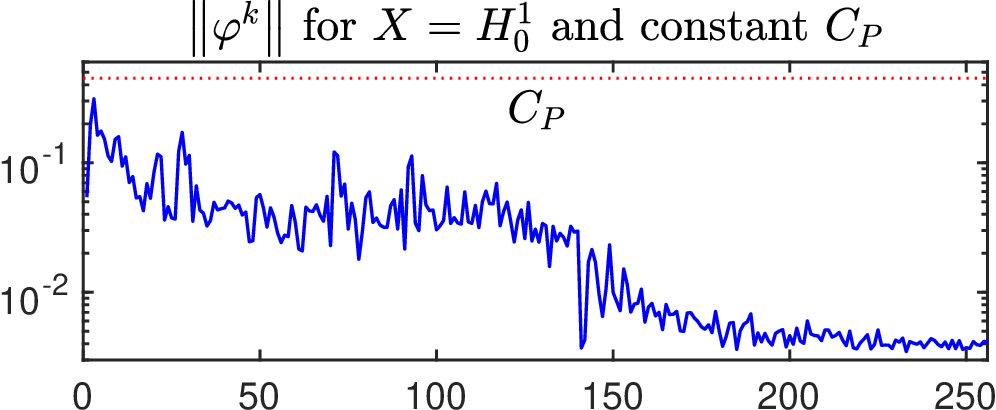}
\begin{caption}{\label{fig:normas}  $L^2$ norms of the elements~$\varphi^1,\ldots,\varphi^J$ of the POD basis for~$X=H^1_0$ and constant~$C_P$ in \eqref{Poincare}.}
\end{caption}
\end{center}
\end{figure}

\section{Abstract results}
\label{se:abstract}
In this section we state and prove a series of results which are needed to prove those in Section~\ref{se:main}. We do this in a more general setting.

Let $X$ be a Hilbert space with inner product $(\cdot,\cdot)$ and associated norm~$\bigl\|\cdot\bigr\|$.
In order to have shorter subindices in this section for $T>0$ and positive integer $M$ we denote
$$
\tau=T/M,
$$
instead of~$\Delta t$ as in previous sections. Thus, we write,
$t_n=n\tau$, $n=0,\ldots,M$. For a sequence $f_\tau=(f_n)_{n=0}^M$ in~$X$ we denote
$$
Df_n=D^1 f_n = \frac{f_n-f_{n-1}}{\tau },\qquad n=1,\ldots,M,
$$
and, for $k=2,\ldots\,J$,
$$
D^k f_n= \frac{D^{k-1} f_n - D^{k-1} f_{n-1}} {\tau},\qquad n=k,\ldots,J.
$$
(Notice the backward notation). Also, for simplicity we will use
$$
D^0 f_n=f_n,\qquad n=0,\ldots,M.
$$
We define
\begin{align}
\label{norma0}
\bigl\|f_\tau \bigr\|_{0}&=\biggl( \tau \sum_{n=0}^{M}\bigl\| f_n\bigr\|^2 \biggr)^{1/2},\\
\label{norma1}
\bigl\|Df_\tau \bigr\|_{0}&=\biggl( \tau \sum_{n=1}^{M}\bigl\| Df_n\bigr\|^2 \biggr)^{1/2},
\end{align}

The proof of the following lemma is the discrete counterpart of the following identity and estimate
$$
\left\|f(t_n)\right\|^2=\left\|f(t_m)\right\|^2 +2 \int_{t_m}^{t_n} (f(s),f'(s))\, \di s \le \left\|f(t_m)\right\|^2 +2 \left\| f\right\| \bigl\|f'\bigr\|.
$$

\begin{lema}\label{le:Agmond} Let $f_\tau=(f_n)_{n=0}^M$ a sequence in~$X$ satisfying that $f_0+\cdots+f_M=0$. Then
$$
\bigl\|f_n\bigr\| \le c_A\bigl\| f_\tau \bigr\|_0^{1/2} \bigl\| Df_\tau \bigr\|_0^{1/2},\qquad n=0,\ldots,M,
$$
where
\begin{equation}
\label{c_A}
c_A = \sqrt{2 + \sqrt{2}/2}.
\end{equation}
\end{lema}
\begin{proof} For $0\le m<n$, we write
\begin{align*}
\bigl\| f_n\bigr\|^2 - \bigl\| f_m\bigr\|^2 &= \bigl\| f_n\bigr\|^2 - \bigl\| f_{n-1}\bigr\|^2 
  +\cdots + \bigl\| f_{m+1}\bigr\|^2 - \bigl\| f_m\bigr\|^2
\nonumber\\
 {}&= \sum_{j=m+1}^n \hbox{\rm Re}((f_j+f_{j-1}, f_j-f_{j-1}))
 \nonumber\\
  {}&\le \tau \sum_{j=m+1}^n \bigl\|f_j+f_{j-1}\bigr\| \bigl\| Df_j\bigr\|.
\end{align*}
Now applying H\"older's inequality and noticing that $\bigl\| a+b\bigr\|^2 \le 2\bigl\|a\bigr\|^2 + 2\bigl\| b\bigr\|^2$,  we have
\begin{equation}
\label{ag0}
\bigl\| f_n\bigr\|^2- \bigl\| f_m\bigr\|^2 \le \biggl(2\tau \bigl\| f_n\bigr\|^2 + 4\tau\sum_{j=m+1}^{n-1} \bigl\| f_j\bigr\|^2+2\tau \bigl\| f_m\bigr\|^2
 \biggr)^{1/2} \bigl\| D f_\tau\bigr\|_0,
\end{equation}
and, hence,
\begin{align}
\label{ag1}
 \bigl\| f_n\bigr\|^2- \bigl\| f_m\bigr\|^2&\le 2\bigl\| f_\tau\bigr\|_0 \bigl\| Df_\tau\bigr\|_0.
\end{align}
A similar argument also shows~\eqref{ag1} for $n<m\le M$.

Also, for $0\le l\le m$, we have $f_m =f_l + \tau (Df_{l+1} + \cdots + Df_m)$, so that applying H\"older's inequality it follows that
\begin{equation}
\label{ag2}
\bigl\| f_m - f_l \bigr\| \le \sqrt{\left|m-l\right|\tau} \bigl\| D_1^0 f_\tau\bigr\|_0.
\end{equation}
A similar argument shows that~\eqref{ag2} also holds for $m<l \le M$. Now, since we are assuming that $f_\tau$ is of zero mean, we can write
$$
f_m =f_m -\frac{1}{M+1}\sum_{l=0}^M f_l= \frac{1}{M+1}\sum_{l\ne m} (f_m -f_l),
$$
and, thus, taking norms, applying~\eqref{ag2} and~H\"older's inequality one gets
$$
\bigl\| f_m\bigr\| \le \frac{\sqrt{M}}{M+1} \biggl( \tau\sum_{l\ne m} \left|m-l\right|\biggr)^{1/2}\bigl\| Df_\tau\bigr\|_0.
$$
Now, a simple calculation shows that the sum above is at most $M(M+1)/2$, so that
\begin{equation}
\label{ag3}
\bigl\| f_m\bigr\| \le \sqrt{T/2} \bigl\| D f_\tau\bigr\|_0.
\end{equation}
We now take $m$ such that
\begin{equation}
\label{ag4}
\bigl\| f_m\bigr\| = \min_{0\le l\le M} \bigl\| f_l\bigr\|.
\end{equation}
We have that
\begin{equation}
\label{ag45}
\bigl\| f_m\bigr\| = \frac{\sqrt{T}}{\sqrt{T}}\bigl\| f_m\bigr\| = \frac{1}{\sqrt{T}} \left( \tau M\bigl\| f_m\bigr\|^2\right)^{1/2} \le \frac{1}{\sqrt{T}} \bigl\| f_\tau\bigr\|_0.
\end{equation}
From this inequality and~\eqref{ag3} it follows that for $f_m$ satisfying~\eqref{ag4} we have
$$
\bigl\| f_m\bigr\|^2 \le \frac{1}{\sqrt{2}} \bigl\|f_\tau\bigr\|_0\bigl\| Df_\tau\bigr\|_0.
$$
This, together with~\eqref{ag1} finishes the proof.
\end{proof}

\begin{remark}\label{re:comillas} \rm Notice that \eqref{ag0} implies~\eqref{ag1} also in the case where $\left\| f_\tau\right\|$ is defined as
\begin{equation}
\label{comillas}
\bigl\|f_\tau \bigr\|_{0}=\biggl(\frac{\tau}{2}\bigl\| f_0\bigr\|^2+ \tau \sum_{n=1}^{M-1}\bigl\| f_n\bigr\|^2+ \frac{\tau}{2}\bigl\| f_M\bigr\|^2 \biggr)^{1/2},
\end{equation}
and that \eqref{ag45} also holds in this case. Thus, Lemma~\ref{le:Agmond} is valid if $\left\| f_\tau\right\|$ is defined as in~\eqref{comillas}.
\end{remark}

We now treat separately the periodic case and the general case.

\subsection{The periodic case}
\label{se:periodic}

In this subsection we assume that $f_M=f_0$ and we extend periodically the sequence $f_\tau$, that is $f_{m} = f_{M+m}$ for $m=\pm 1,\pm 2,\ldots\,$. Also,
we will consider~$\left\| f_\tau\right\|_0$ as defined in~\eqref{comillas}. We
define~$\bigl\|D^{k} f_\tau\bigr\|_{0}$ as
\begin{equation}
\label{normap}
\bigl\| D^kf_\tau\bigr\|_{0} =\biggl(\tau\sum_{n=1}^M \bigl\| D^k f_n\bigr\|^2\biggr)^{1/2},\quad k=2,\ldots,M.
\end{equation}
Notice that for $k=1$ the expression above coincides with $\left\| D f_\tau\right\|$ defined in~\eqref{norma1}. Also,  since  $f_0=f_M$, the
expression~in~\eqref{normap} for  $k=0$  coincides with
the alternative definition of~$\left\| f_\tau\right\|$ given in~\eqref{comillas}, for which, as observed in~Remark~\ref{re:comillas} 
 Lemma~\ref{le:Agmond} is also valid. 
 
 The proof of the following lemma is the discrete counterpart of integration by parts in a periodic function
 $$
 \bigl\| f'\bigr\|^2 = \int_0^T (f'(t),f'(t))\,\di t = -\int_0^T (f(t),f''(t))\,\di t \le \left\| f\right\| \bigl\| f''\bigr\|.
 $$
 
\begin{lema}\label{le:tresd}
Let $f_\tau=(f_n)_{n\in{\mathbb Z}}$ be a periodic sequence in~$X$ with period $M$. Then, for $k=2,\ldots,M-1$,
$$
\bigl\| D^kf_\tau\bigr\|_0 \le \bigl\|D^{k-1} f_\tau \bigr\|_0^{1/2}  \bigl\|D^{k+1} f_\tau \bigr\|_0^{1/2}
$$
\end{lema}
\begin{proof} Since $\tau \bigl\| D^k f_n\bigr\|^2=\tau (D^k f_n, D^k f_n)= (D^k f_n,D^{k-1} f_n - D^{k-1} f_{n-1})$, we have
\begin{align*}
\bigl\| D^kf_\tau\bigr\|_0^2 &= \sum_{n=1}^M (D^k f_n, D^{k-1} f_n - D^{k-1} f_{n-1})
\nonumber\\
&{}= (D^kf_M ,D^{k-1}f_M) - (D^kf_1,D^{k-1} f_{0})
- \sum_{n=1}^{M-1} (D^k f_{n+1} - D^kf_n, D^{k-1}f_n).
\end{align*}
Now we notice that for the second term on the right-hand side above we have  $(D^kf_1,D^{k-1} f_{0})=(D^kf_{M+1},D^{k-1} f_{M})$, and thus
$$
\bigl\| D^kf_\tau\bigr\|^2 =  - \sum_{n=1}^{M} (D^k f_{n+1} - D^kf_n, D^{k-1}f_n) =-\tau  \sum_{n=1}^{M} (D^{k+1} f_{n+1} , D^{k-1}f_n)
$$
and the proof is finished by applying H\"older's inequality.
\end{proof}

The proof of the following reuslt follows the ideas in the proof of~\cite[Lemma~4.12]{Adams} adapted to the discrete case.
\begin{lema}\label{le:uno_m}
 Let $f_\tau=(f_n)_{n\in{\mathbb Z}}$ a periodic sequence in~$X$ with period $M$. Then, for $m=2,\ldots,M$,
\begin{equation}
\label{uno_m}
\bigl\| D^1 f_\tau\bigr\|_0 \le \bigl\|f_\tau\bigr\|_0^{\frac{m-1}{m}} \bigl\| D^mf_\tau\bigr\|_0^{\frac{1}{m}}.
\end{equation}
\end{lema}
\begin{proof}
We will prove below that~\eqref{uno_m} will be a consequence of the following inequality
\begin{equation}
\label{m-1_m}
\bigl\| D^{k} f_\tau\bigr\|_0 \le \bigl\|f_\tau\bigr\|_0^{\frac{1}{k+1}} \bigl\| D^{k+1}f_\tau\bigr\|_0^{\frac{k}{k+1}}, \qquad k=1,\ldots,M-1,
\end{equation}
which we will now proof by induction. Notice that \eqref{m-1_m} holds for $k=1$ as a direct application of~Lemma~\ref{le:tresd}. Now assume
that~\eqref{m-1_m} holds for $k=1,\ldots,m-1$, and let us check that it also holds for $k=m$. Applying Lemma~\ref{le:tresd} we have
$$
\bigl\| D^{m} f_\tau\bigr\|_0 \le \bigl\| D^{m-1} f_\tau\bigr\|_0^{1/2} \bigl\| D^{m+1}f_\tau\bigr\|_0^{1/2}.
$$
Applying the induction hypothesis we have
$$
\bigl\| D^{m} f_\tau\bigr\|_0 \le \bigl\| f_\tau\bigr\|_0^{\frac{1}{2m}}\bigl\| D^{m} f_\tau\bigr\|_0^{\frac{m-1}{2m}} \bigl\| D^{m+1}f_\tau\bigr\|_0^{1/2},
$$
so that dividing by~$\bigl\| D^{m} f_\tau\bigr\|_0^{\frac{m-1}{2m}}$ we obtain
$$
\bigl\| D^{m} f_\tau\bigr\|_0^{\frac{m+1}{2m}} \le \bigl\| f_\tau\bigr\|_0^{\frac{1}{2m}} \bigl\| D^{m+1}f_\tau\bigr\|_0^{1/2},
$$
from where~\eqref{m-1_m} follows.

Now, to proof~\eqref{uno_m}, we start with the case $m=2$, which holds as a direct application of Lemma~\ref{le:tresd} for $k=1$. Then we apply~\eqref{m-1_m}
repeatedly to the factor with the highest-order difference, that is
$$
\displaylines{
\bigl\| D^1 f_\tau\bigr\|_0 \le \bigl\|f\bigr\|^{\frac{1}{2}} \bigl\| D^2f_\tau\bigr\|_0^{\frac{1}{2}} \le \bigl\|f\bigr\|^{\frac{1}{2}+\frac{1}{6}}  \bigl\| D^3f_\tau\bigr\|_0^{\frac{1}{3}} 
\le  \bigl\|f\bigr\|^{\frac{1}{2}+\frac{1}{6} + \frac{1}{12}}  \bigl\| D^4f_\tau\bigr\|_0^{\frac{1}{4}}\le {} \cr
{} \ldots \le 
 \bigl\|f\bigr\|^{\frac{1}{2}+\frac{1}{6} +\cdots  \frac{1}{m(m-1)}}  \bigl\| D^mf_\tau\bigr\|_0^{\frac{1}{m}},
 }
 $$
 and the proof is finished by noticing that 
\begin{align*}
\frac{1}{2}+\frac{1}{6} +\cdots  \frac{1}{m(m-1)} &= \frac{1}{2} + \left( \frac{1}{2} - \frac{1}{3} \right) 
 + \left( \frac{1}{3} - \frac{1}{4} \right) 
+\cdots+
 \left( \frac{1}{m-1} - \frac{1}{m} \right)
 \nonumber\\
 {}&\le 1 - \frac{1}{m} = \frac{m-1}{m}.
 \end{align*}
 \end{proof}
 
From Lemmas~\ref{le:Agmond} and~\ref{le:uno_m} the following result follows.
\begin{lema}\label{le:esti_max}
 Let $f_\tau=(f_n)_{n\in{\mathbb Z}}$ a periodic sequence in~$X$ with period $M$ and zero mean. Then, for $k=1,\ldots,M$, the following estimate holds:
\begin{equation}
\label{est_max}
\bigl\| f_n\bigr\| \le c_A\bigl\|f_\tau\bigr\|_0^{1-\frac{1}{2m}} \bigl\| D^mf_\tau\bigr\|_0^{\frac{1}{2m}},
\end{equation}
where $c_A$ is the constant in~\eqref{c_A}.
\end{lema}

\subsection{The general case}

In this section, and unless stated otherwise, $f_\tau=(f_n)_{n=0}^M$ is a sequence in~$X$ with zero mean. We define
\begin{align*}
\bigl\| D^k f_\tau\bigl\|_0 &= \biggl(\tau \sum_{n=k}^M \bigl\| D^k f_n\bigr\|^2\biggr)^{1/2}, \qquad k=1,\ldots,M,
\nonumber\\ 
\end{align*}
and, in order to obtain scale invariant constants, for each $k=0,\ldots,M$,
\begin{equation}
\label{norma_mg}
\bigl\| D^k f_\tau\bigl\|_m = \biggl(\sum_{j=k}^{k+m} \frac{1}{T_j^{2(m+k-j)}}\bigl\| D^j f_\tau\bigr\|_0^2\biggr)^{1/2}, \qquad m=1,\ldots, M-k,
\end{equation}
where, here and in the sequel, $T_0=T$, and
\begin{equation}
\label{T_k}
T_k = (M+1-k) \tau,\qquad k=1,\ldots,M.
\end{equation}

We also denote
\begin{equation}
\label{m_k}
m_k=\frac{1}{M+1-k} \sum_{n=k}^M D^kf_n.
\end{equation}
Observe that by applying H\"older's inequality one gets
\begin{equation}
\label{norma_m_k}
\bigl\| m_k\bigr\| \le \frac{1}{\sqrt{T_k}} \bigl\| D^k f_\tau\bigr\|_0
\end{equation}
Consequently, for the sequence $D^k f_\tau -m_k = (D^k f_n -m_k)_{n=k}^M$, we have
\begin{equation}
\label{norma_0Dk-mk}
\bigl\| D^kf_\tau - m_k\bigr\|_0 \le 2\bigl\| D^k f_\tau\bigr\|_0.
\end{equation}

Our first result 
extends
Lemma~\ref{le:Agmond} to the sequences~$(D^kf_n)_{n=k}^M$, whose means are not 0.

\begin{lema}\label{le:AgmonDk} The following bound holds for $k=1,\ldots M$,
$$
\bigl\| D^k f_n\bigr\| \le c_{A,1} \bigl\| D^k f_\tau\bigr\|_0^{1/2}\bigl\| D^k f_\tau\bigr\|_1^{1/2},\qquad n=k,\ldots,M-1,
$$
where
\begin{equation}
\label{c_Ak}
c_{A,1} = \left(1 + (\sqrt{2}c_A)^{4/3} \right)^{3/4},
\end{equation}
and $c_A$ isn the constant in~\eqref{c_A}.
\end{lema}

\begin{proof} Applying Lemma~\eqref{le:Agmond} we have
\begin{align}
\label{ag5}
\bigl\| D^kf_n - m_k\bigr\| &\le c_A  \bigl\| D^kf_\tau - m_k\bigr\|_0^{1/2} \bigl\| D^{k+1} f_\tau\bigr\|_0^{1/2}
\nonumber\\
&{} \le \sqrt{2}c_A  \bigl\| D^kf_\tau\bigr\|_0^{1/2} \bigl\| D^{k+1} f_\tau\bigr\|_0^{1/2},
\end{align}
where in the last inequality we have applied~\eqref{norma_0Dk-mk}. Thus, by writing $D^k f_n = D^k f_n -m_k + m_k$ and in view of~\eqref{norma_m_k}, we have
\begin{align*}
\bigl\| D^kf_n\bigr\| &\le \sqrt{2}c_A  \bigl\| D^kf_\tau\bigr\|_0^{1/2} \bigl\| D^{k+1} f_\tau\bigr\|_0^{1/2} +  T_k^{-1/2}\bigl\| D^k f_\tau\bigr\|_0
\nonumber\\
&{}=  \bigl\| D^kf_\tau\bigr\|_0^{1/2}\left(\sqrt{2}c_A \bigl\| D^{k+1} f_\tau\bigr\|_0^{1/2}+ T_k^{-1/2}\bigl\| D^k f_\tau\bigr\|_0^{1/2}\right).
\end{align*}
Now applying H\"older's inequality, $ab + cd \le (a^p+c^p)^{\frac{1}{p}} (b^q+d^q)^{\frac{1}{q}}$, to the second factor above, with $p=4/3$ and~$q=4$,
the proof is finished.
\end{proof}

Next, we extend Lemma~\ref{le:tresd} to the general case

\begin{lema}\label{le:tresdg}
For $k=1,\ldots,M-1$,
$$
\bigl\| D^kf_\tau\bigr\|_0 \le c_{B,1}\bigl\|D^{k-1} f_\tau \bigr\|_0^{1/2}  \bigl\|D^{k} f_\tau \bigr\|_1^{1/2},
$$
where 
\begin{equation}
\label{c_B}
c_{B,1} =2\left(1+ 2(c_Ac_{A,1})^2\right)^{1/2}.
\end{equation}
In the case $k=1$, $c_{B,1}$ can be replaced by $c_{B,1}/\sqrt{2}$.
\end{lema}

\begin{proof} We notice that $D_k f_\tau = D(D^{k-1}f_\tau - m_{k-1})$ and let us denote
$$
y_n=D^{k-1} f_n -m_{k-1},\qquad n=k-1,\ldots,M.
$$
Thus, arguing as in the proof of Lemma~\ref{le:tresd} we have
\begin{align*}
\bigl\| D^k f_\tau\bigr\|^2  &=\sum_{n=k}^M (D^k f_n, y_n-y_{n-1})
\nonumber\\
&{} =(D^kf_M,y_M)-(D^kf_k,y_{k-1}) -\sum_{n=k}^{M-1} (D^{k}f_{n+1}-D^k f_n ,y_n)
\nonumber\\
&{} =(D^kf_M,y_M)-(D^kf_k,y_{k-1}) -\tau\sum_{n=k}^{M-1} (D^{k+1}f_{n+1} ,y_n).
\end{align*}
 We apply~\eqref{ag5} and~Lemma~\ref{le:AgmonDk}
to the first two terms on  right-hand side above, and we apply~\eqref{norma_0Dk-mk} to the last one, to get,
\begin{align*}
\bigl\| D^k f_\tau\bigr\|^2 &\le  2\sqrt{2}c_Ac_{A,1}  \bigl\| D^{k-1}f_\tau\bigr\|_0^{\frac{1}{2}} \bigl\| D^{k} f_\tau\bigr\|_0\bigl\| D^{k} f_\tau\bigr\|_1^{\frac{1}{2}}
 + 2\bigl\| D^{k-1}f_\tau \bigr\|_0
\bigl\| D^{k+1}f_\tau \bigr\|_0
\nonumber\\
{}
&{}\le \frac{1}{2} \bigl\| D^{k} f_\tau\bigr\|_0^2 + 4c_A^2c_{A,1}^2 \bigl\| D^{k-1}f_\tau\bigr\|_0\bigl\| D^{k} f_\tau\bigr\|_1
+ 2\bigl\| D^{k-1} f_\tau\bigr\|_0
\bigl\| D^{k+1}f_\tau \bigr\|_0.
\end{align*}
Thus,
\begin{align*}
\bigl\| D^k f_\tau\bigr\|^2 &\le 4\bigl\| D^{k-1} f_\tau\bigr\|_0\left(2c_A^2c_{A,1}^2\bigl\| D^{k} f_\tau\bigr\|_1 + \bigl\| D^{k+1}f_\tau \bigr\|_0\right)
\nonumber\\
&{}\le 4\bigl\| D^{k-1} f_\tau\bigr\|_0\left(2c_A^2c_{A,1}^2+1\right)\bigl\| D^{k} f_\tau\bigr\|_1,
\end{align*}
and the proof is finished for $k\ge 2$. For $k=1$ the argument is the same but taking into account that $m_0=0$.
\end{proof}

\begin{lema}\label{le:uno_mg}
For $m=2,\ldots,M$,
\begin{equation}
\label{uno_mg}
\bigl\| D^1f_\tau\bigr\|_0 \le c_{m}\bigl\|f_\tau \bigr\|_0^{\frac{m-1}{m}}  \bigl\|D^{1} f_\tau \bigr\|_{m-1}^{\frac{1}{m}},
\end{equation}
where
\begin{equation}
\label{c_m}
c_{m} = \prod_{j=0}^{m-2} \hat c_j^\frac{1}{j+1}.
\end{equation}
and the constants~$\hat c_j$ are given in~\eqref{hatc_j} below.
\end{lema}
\begin{proof}
The result will be a consequence of the following inequalities that will be proven below by induction:
\begin{align}
\label{interp_3a}
\bigl\|    D ^1f_\tau\bigr\|_{j} &\le \hat c_{j} \bigl\|     f_\tau\bigr\|_{0}^{\frac{1}{j+2}}  \bigl\|    D ^1 f_\tau\bigr\|_{j+1}^{\frac{j+1}{j+2}}.
\\
\label{interp_3b}
\bigl\|    D ^kf_\tau\bigr\|_{j} &\le \hat d_{j} \bigl\|     D ^{k-1} f_\tau\bigr\|_{0}^{\frac{1}{j+2}}  \bigl\|    D^{k} f_\tau\bigr\|_{j+1}^{\frac{j+1}{j+2}}.
\end{align}
If~\eqref{interp_3a} holds for $j=0,\ldots,m-2$, then, applying it successively we have
$$
\displaylines{
\bigl\|    D ^1f_\tau\bigr\|_{0} \le \hat c_0 \bigl\|   f_\tau \bigr\| _{0}^{\frac{1}{2}}\bigl\|    D ^1f_\tau\bigr\|_{1}^{\frac{1}{2}}\le
 \hat c_0\hat c_1^{\frac{1}{2}} \bigl\|   f_\tau \bigr\| _{0}^{\frac{1}{2}+\frac{1}{6}}\bigl\|    D ^1f_\tau\bigr\|_{2}^{\frac{1}{3}}
 \le \ldots 
 \cr
 \le c_m \bigl\|   f_\tau \bigr\| _{0}^\beta\bigl\|    D ^1f_\tau\bigr\|_{m-1}^{\frac{1}{m}}
 }
$$
where
\begin{equation}
\label{la_beta}
\beta=\sum_{j=1}^{m-1} \frac{1}{j(j+1)}=\sum_{j=1}^{m-1}\left(\frac{1}{j} - \frac{1}{j+1}\right) = 1 -\frac{1}{m}= \frac{m-1}{m},
\end{equation}
so that \eqref{uno_mg} follows.
We now prove~\eqref{interp_3a} and~\eqref{interp_3b}. Lemma~\ref{le:tresdg} shows they~hold for $j=0$ with $\hat d_0 = c_{B,1}$ and $\hat c_0 = c_{B,1}/\sqrt{2}$. Let us assume
that they hold for~$j=0,\ldots,m-2$, and let us show that they hold for~$m-1$.
We notice
$$
\bigl\|    D ^1f_\tau\bigr\|_{m-1} = \left( T^{-2(m-1)}\bigl\|    D ^1f_\tau\bigr\|_{0}^2 +\bigl\|    D ^2f_\tau\bigr\|_{m-2}^2 \right)^{\frac{1}{2}}
$$
and apply the induction hypotheses to the second term on the right-hand side above to get
\begin{align*}
\bigl\|    D ^1f_\tau\bigr\|_{m-1}&\! =\! \left( T^{-2(m-1)}\bigl\|    D ^1f_\tau\bigr\|_{0}^2 +\hat d_{m-2}^2\bigl\|    D ^1f_\tau\bigr\|_{0}^{\frac{2}{m}}
\bigl\|    D ^2f_\tau\bigr\|_{m-1}^{\frac{2(m-1)}{m}} \right)^{\frac{1}{2}}
\nonumber\\
&{}\! =\! 
\bigl\|    D ^1f_\tau\bigr\|_{0}^{\frac{1}{m}}
 \!\left( T^{-2(m-1)}\bigl\|    D ^1f_\tau\bigr\|_{0}^{\frac{2(m-1)}{m}} \!+\hat d_{m-2}^2 \bigl\|    D ^2f_\tau\bigr\|_{m-1}^{\frac{2(m-1)}{m}} \right)^{\frac{1}{2}}
 \nonumber\\
 &{}\! =\! AB
\end{align*}
Now, we use that, as argued above, the induction hypothesis implies that~\eqref{uno_mg} holds, so that we can write
\begin{align*}
\bigl\|    D ^1f_\tau\bigr\|_{m-1}& = c_m^{\frac{1}{m}}
\bigl\| f_\tau\bigr\|_0^{\frac{m-1}{m^2}}
\bigl\|    D ^1f_\tau\bigr\|_{m-1}^{\frac{1}{m^2}} B,
\end{align*}
so that, dividing by~$\bigl\|    D ^1f_\tau\bigr\|_{m-1}^{\frac{1}{m^2}}$ we have
\begin{align*}
\bigl\|    D ^1f_\tau\bigr\|_{m-1}^{\frac{m^2-1}{m^2}}& = c_m^{\frac{1}{m}}
\bigl\| f_\tau\bigr\|_0^{\frac{m-1}{m^2}}
 B,
\end{align*}
which implies
\begin{align}
\label{uno_mg1}
\bigl\|    D ^1f_\tau\bigr\|_{m-1}& = c_m^{\frac{m}{m^2-1}}
\bigl\| f_\tau\bigr\|_0^{\frac{1}{m+1}}
 B^{\frac{m^2}{m^2-1}}.
\end{align}
Now we apply Holder's inequality $ab+cd \le (a^p+c^p)^{1/p}(b^q+d^q)^{1/q}$ to~$B$, with~$p=m$ and~$q=m/(m-1)$
\begin{align*}
B &\le \left(1+\hat d_{m-2}^{2m}\right)^{\frac{1}{2m}} \left( T^{-2m}\bigl\|    D ^1f_\tau\bigr\|_{0}^{2}+  \bigl\|    D ^2f_\tau\bigr\|_{m-1}^2 \right)^{\frac{m-1}{2m}}
\nonumber\\
&{}=  \left(1+\hat d_{m-2}^{2m}\right)^{\frac{1}{2m}} \bigl\| D^1 f_\tau\bigr\|_m^{\frac{m-1}{m}},
\end{align*}
which, together with~\eqref{uno_mg1}, implies
$$
\bigl\|    D ^1f_\tau\bigr\|_{m-1} = \left( \left(1+\hat d_{m-2}^{2m}\right)c_m^2\right)^{\frac{m}{2(m+1)(m-1)}}
\bigl\| f_\tau\bigr\|_0^{\frac{1}{m+1}}\bigl\| D^1 f_\tau\bigr\|_m^{\frac{m}{m+1}}.
$$
This is~\eqref{interp_3a} for $j=m-1$ with
\begin{equation}
\label{hatc_j} \hat c_j = \left(1+\hat d_{j-1}^{2(j+1)}\right)^{\frac{j+1}{2(j+2)j}} \biggl(\prod_{i=0}^{j-1} \hat c_i^{\frac{1}{i+1}}\biggr)^{\frac{j+1}{(j+2)j}},\qquad \hat c_0=c_{B,1}/\sqrt{2},
\end{equation}
where $c_{B,1}$ is the constant given in~\eqref{c_B}, and, with the same arguments, one can show that the constants $\hat d_j$ are given by
\begin{equation}
\label{hatdjj} \hat d_j = \left(1+\hat d_{j-1}^{2(j+1)}\right)^{\frac{j+1}{2(j+2)j}} \biggl(\prod_{i=0}^{j-1} \hat d_i^{\frac{1}{i+1}}\biggr)^{\frac{j+1}{(j+2)j}},
\qquad \hat d_0=c_{B}.
\end{equation}
\end{proof}

Table~\ref{tablacm} shows the first nine values of~$c_m$, and table~\ref{tablacm2} further values.
\begin{table}[h]
\begin{center}
$$
\begin{array}{|c|c|c|c|c|c|c|c|c|c|c|c|}
\hline
m  &  2 &  3 &  4 &  5 &  6 &  7 &  8 &  9 & 10 \\ \hline
c_m & 9.558 &  33.17 &  67.26 & 103.7 & 137.7 & 167.5 & 193.0 & 214.7 & 233.4  \\
\hline
\end{array}
$$
\caption{\label{tablacm} The first nine values of coefficient $c_m$ in~\eqref{c_m}.}
\end{center}
\end{table}
\begin{table}[h]
\begin{center}
$$
\begin{array}{|c|c|c|c|c|c|c|c|c|c|c|c|}
\hline
m\vphantom{\big|^{2'}}  &  10^2 &  10^3 &  10^4 & 10^5  \\ \hline
c_m & 432.7 &  458.5 &  461.1 & 461.4 \\
\hline
\end{array}
$$
\caption{\label{tablacm2} Furhter values of coefficient $c_m$ in~\eqref{c_m}.}
\end{center}
\end{table}

From Lemmas~\ref{le:Agmond} and~\ref{le:uno_mg} the following result follows. 

\begin{lema}\label{le:esti_maxg}
 Let $f_\tau=(f_n)_{n=0}^M$ a sequence in~$X$ with zero mean. Then, for $m=1,\ldots,M-1$, the following estimate holds:
\begin{equation}
\label{est_maxg}
\bigl\| f_n\bigr\| \le c_Ac_m^{\frac{1}{2}}\bigl\|f_\tau\bigr\|_0^{1-\frac{1}{2m}} \bigl\| D^1f_\tau\bigr\|_{m-1}^{\frac{1}{2m}},
\end{equation}
where $c_A$ is the constant in~\eqref{c_A} and~$c_m$ is $c_m=1$ if $m=1$ and the constant~in~\eqref{c_m} for $m\ge 2$.
\end{lema}

%
%
%
\subsection{Application to sequences of funtion values}
We consider now the case where $X$ is either $L^2(\Omega)$ or some Sobolev space~$H^s(\Omega)$ for some domain~$\Omega \subset {\mathbb R}^d$,
$\left\|\cdot\right\|$ its norm, and $f_n=f(\cdot, t_n)$ for some function $f:\Omega \times [0,T] \rightarrow {\mathbb R}$ such that $f\in H^m(0,T,X)$. We first notice that since for $x\in \Omega$ we have $D^{k-1} f(x,t_n) =\partial^{k-1} f(x,\xi)$ for some~$\xi\in (t_{n-k+1},t_n)$, then it follows that
\begin{equation}
\label{estiDk}
\bigl\| D_kf_n\bigr\| \le \frac{\sqrt{k}}{\sqrt{\tau}} \bigl\| \partial^kf\bigr\|_{L^2(t_{n-k},t_n,X)},\qquad k\le n\le M.
\end{equation}
Consequently, 
\begin{equation}
\bigl\| D^kf_\tau\bigr\|_0 \le k\bigl\| \partial^kf\bigr\|_{L^2(0,T,X)}
\label{d-c-k}
\end{equation}
We have the following result.

\begin{Theorem}\label{th:cotas_f} Let $f:X\times {\mathbb R}\rightarrow {\mathbb R}$ a  function which is $T$-periodic in its last variable and such that $f\in H^m(0,T,X)$. Then, the sequence $f_\tau=(f_n)_{n=1}^M$ satisfies the following bounds for $1\le m\le M-1$:
\begin{align*}
\bigl\| Df_\tau\bigr\|_0 &\le 2^{\frac{m-1}{m}}m^{\frac{1}{m}}\left\| f_\tau\right\|_0^{\frac{m-1}{m}} \left\|\partial^m f\right\|_{L^2(0,T,X)}^{\frac{1}{m}},
\nonumber\\
\bigl\| f_n\bigr\| &\le 2^{\frac{m-1}{2m}}m^{\frac{1}{2m}}c_A\left\| f_\tau\right\|_0^{1-\frac{1}{2m}} \left\|\partial^m f\right\|_{L^2(0,T,X)}^{\frac{1}{2m}}
+\frac{1}{\sqrt{T}}\bigl\| f_\tau\bigr\|_0, \quad 1\le n\le M,
\end{align*}
where $c_m$ is the constant~in Lemma~\ref{le:esti_maxg}.
If $f_1+\cdots+f_M=0$, the the last term on the right-hand side above can be omitted.
\end{Theorem}
\begin{proof} Let us denote $m_0=(f_1+\cdots+f_M)/M$. Applying H\"olders inequality is easy to check that $\left| m_0\right| \le T^{-1/2}\bigl\| f_\tau\bigr\|_0$ and, consequently,
$\left\| f_\tau-m_0\right\|_0 \le 2\left\| f_\tau\right\|_0$. Thus, by expressing $f_n=(f_n-m_0)+m_0$, applying  Lemmas~\ref{le:uno_m} and~\ref{le:esti_max} to $f_\tau - m_0$ and using~\eqref{d-c-k} the result follows for $m\ge 2$. For $m=1$ the result follows form Lemma~\ref{le:esti_max} and~\eqref{d-c-k}.
\end{proof}

For the general case, in view of the definition of~$\bigl\| D^k f_\tau\bigl\|_m$ in~\eqref{norma_mg} and~\eqref{d-c-k} we  have
\begin{align}
\label{d-c-kg0}
\bigl\|D f_\tau \bigr\|_{m-1}^{\frac{1}{m}} &\!\le \!  \biggl( \sum_{k=0}^{m-1} (k+1)^2\left(\frac{T}{T_{k+1}}\right)^{2(m-k)} \!\!\!
\frac{1}{T^{2(m-k)}}\bigl\| \partial_t^{k+1} f\bigr\|_{L^2(0,T,X)}^2\biggr)^{\frac{1}{2m}}
\nonumber\\
&{}\!\le \!   \max_{0\le k \le m-1} (k+1)^{\frac{1}{m}}\left(\frac{T}{T_{k+1}}\right)^{\frac{m-k}{m}} \bigl\| \partial_tf\bigr\|_{H^{m-1}(0,T,X)}^{\frac{1}{m}},
\end{align}
where~$\left\|\cdot\right\|_{H^{m-1}(0,T,X)}$ is defined in~\eqref{norma_Hm0TX}.
We now estimate the maximum above. In view of the expression of $T_k$ in~\eqref{T_k}, for $1\le k\le m-1$ we have
$$
\displaylines{
\left(\frac{T}{T_{k+1}}\right)^{\frac{m-k}{m}} =\left(\frac{M}{M-k}\right)^{\frac{m-k}{m}} =
\left(1 +\frac{k}{M-k}\right)^{\frac{m-k}{m}}  =
\left(1 +\frac{1}{(M/k)-1}\right)^{\frac{m-k}{m}}  
\cr
{}=\left(\left(1 +\frac{1}{(M/k)-1}\right)^{(M/k)-1}\right)^{\frac{m-k}{m((M/k)-1)}} 
\le e^{\frac{m-k}{m((M/k)-1)}} 
=
e^{\frac{k(m-k)}{m(M-k)}}.
}
$$
Now the function $x\mapsto x(m-x)/(m(M-x))$ is monotone increasing in the interval~$[1,m-1]$, unless $m\le M/4$ where it have a relative maximum at
at $(M/2)-\sqrt{(M^2/4)-Mm}$, but is easy to check that this value is strictly larger than $m$. Consequently, its maximum is achieved at $x=m-1$. Noticing that
$x^{1/x} \le e^{\frac{1}{e}}$
 for $x>0$, we conclude that, for $2\le m\le M$ 
\begin{align*}
\max_{0\le k \le m-1} (k+1)^{\frac{1}{m}}\left(\frac{T}{T_{k+1}}\right)^{\frac{m-k}{m}} \le e^{\frac{m-1}{m(M-m+1)}+\frac{1}{e}}\le e^{1+\frac{1}{e}}.
\end{align*}
One can check that $e^{1+\frac{1}{e}}\le 3.93$.
Thus, from~\eqref{d-c-kg0} it follows that
\begin{equation}
\label{d-c-kg}
\bigl\|D^{1} f_\tau \bigr\|_{m-1}^{\frac{1}{m}} \le 4 \bigl\| \partial_tf\bigr\|_{H^{m-1}(0,T,X)}^{\frac{1}{m}}.
\end{equation}

\begin{Theorem}\label{th:cotas_f_g} Let  $f\in H^m(0,T,X)$.  Then, the sequence~$f_\tau=(f_n)_{n=0}^M$  satisfies the following bound holds for $1\le m\le M$:
\begin{eqnarray}
\label{cota_Df_g_D1}
\bigl\| D^1 f_\tau\bigr\| &\le 2^{\frac{m-1}{m}}c_m \left\| f_\tau\right\|_0^{\frac{m-1}{m}} \left\| \partial_t f\right\|_{H^{m-1}(0,T,X)}^{\frac{1}{m}},
\\
\bigl\| f_n\bigr\| &\le 2^{\frac{m-1}{2m}}  c_Ac_m^{1/2} \bigl\| f_\tau\bigr\|_0^{1-\frac{1}{2m}} \left\| \partial_t f\right\|_{H^{m-1}(0,T,X)}^{\frac{1}{2m}}
+\frac{1}{\sqrt{T}}\bigl\| f_\tau\bigr\|_0,
\nonumber\\
&\hphantom{{}\le 2^{frac{m-1}{2m}} c_A\left(c_m\exp(1)\right)^{1/2}} 
   n=0,\ldots,M.
\label{cota_Df_g}
\end{eqnarray}
If $f_0+\cdots+f_M=0$, the last term on the right-hand side above can be omitted.
\end{Theorem}
\begin{proof} Let us denote $m_0=(f_0+\cdots+f_M)/(M+1)$. Applying H\"olders inequality one sees that $\left| m_0\right| \le T^{-1/2}\bigl\| f_\tau\bigr\|_0$ so that
$\left\| f_\tau-m_0\right\|_0 \le 2\left\| f_\tau\right\|_0$. Thus, by writing $f_n=(f_n-m_0)+m_0$, applying  Lemmas~\ref{le:uno_mg} (if $m\ge 2$) and~\ref{le:esti_maxg} to $f_\tau - m_0$ and using~\eqref{d-c-kg} the proof is finished. 
\end{proof}


\section{Summary and conclusions}
Several estimates, pointwise and averaged, have been proved for POD methods whose data set does not include DQs of the snapshots. These estimates include pointwise projection errors
of snapshots (Theorems~\ref{th:1} and~\ref{th:2} in the general case and estimate~\eqref{esti1} in the case of periodic functions in time) for orthogonal projections in $L^2$ and~$H^1_0$. They also include $L^2$ norm of $H^1_0$-projection errors (estimate~\eqref{diferente1}), for which estimate~\eqref{different_norms} has been used. Our estimates also include (the square root of) the average of squared projection error of DQs (Theorem~\ref{th:3} for the general case and estimate~\eqref{esti2} for periodic functions) again in the~$L^2$ and~$H^1_0$ cases, as well as the average of squares of $L^2$ norms of Ritz projection errors of~DQs (estimate~\eqref{diferente2}). In all cases, when projections are onto the space spanned by the first $r$ elements of the POD basis, the rate of decay in terms of $\sigma_{r+1}^2 + \cdots+\sigma_J^2$ is as close to the optimal value $(\sigma_{r+1}^2 +\cdots+ \sigma_J^2)^{1/2}$ as the smoothness of the function from where the snapshots are taken allows.

These estimates have been used to obtain error bounds with the same rates of decay for POD methods applied to the heat equation, where both backward Euler method and the BDF2 are considered for discretization of the time variable (Theorem~\ref{th:heat}). The heat equation has been chosen for simplicity, but it is clear that, with techniques for nonlinear problems such as those, for example, in~\cite{nos_siam,nos_semi,iliescu_wang,koc_chacon-rubino,Kunisch-Volkwein2},
 it is possible to extend these estimates to POD methods not using DQs in the data set for nonlinear equations.
 
 Numerical experiments are also presented in this paper, where snapshots are taken from a challenging problem due to the size of derivatives with respect to time~(see Fig.~\ref{fig:derivs}), derivatives which feature in the estimates commented above. In these experiments, pointwise projection errors are overestimated by our error bounds by a factor not exceeding~9, and by a factor not exceeding 8.5 in the~case of the (square root of) the average of squares of $L^2$ norms of Ritz projection errors of DQs, if the estimates are those in~Remark~\ref{re:diferente_POD}


This research was motivated by a better understanding of POD methods that do not include DQs of the snapshots in the data set. In particular by the apparent contradiction that while counterexamples in~\cite{samu_et_al} clearly show that pointwise projection errors do degrade with the number of snapshots, computations like those in~Table~\ref{tab_T1r25} suggest that this may not necessarily be the case in practice. By taking into account the smoothness of functions from where the snapshots are taken, new estimates have been obtained and the apparent contradiction mentioned above has been explained, as well as better knowledge of the decay rate of pointwise errors when DQs are omitted from the data set.


\begin{thebibliography}{99}

\bibitem{Adams}
{R.~A. Adams}, {\em Sobolev spaces}, Academic Press [A subsidiary of
  Harcourt Brace Jovanovich, Publishers], New York-London, 1975.


\bibitem{Constantin-Foias} P. Constantin and C. Foias, {\em Navier-Stokes Equations\/}. Chicago Lectures on Mathematics. The University of Chicago Press, Chicago, 1988.
\bibitem{chapelle_et_al} D. Chapelle, A. Gariah and J. Sainte-Marie, 
\newblock Galerkin approximation with proper orthogonal decomposition: new error estimates and illustrative examples. 
\newblock {ESAIM: M2AN}, 46 (2012) 731-757.


\bibitem{locke-singler}
S.L.~Eskew and J. R.~Singler.
\newblock A new approach to proper orthogonal decomposition with difference quotients.
\newblock {\em Adv. Comput. Math.}, 49(2) (2023.), Paper No. 13.

\bibitem{nos_siam}
B. Garc\'{i}a-Archilla, V. John and J. Novo.
\newblock POD-ROMs for incompressible flows including snapshots of the temporal derivative of the full order solution.
\newblock {\em SIAM. J. Numer. Anal.}, 61 (3) (2023), 1340-1368.


\bibitem{nos_ap_let}
B. Garc\'{i}a-Archilla, V. John and J. Novo.
\newblock Second order error bounds for POD-ROM methods based on first order divided differences.
\newblock {\em Appl. Math. Letters}, 143 (2023), Paper No. 108836.

\bibitem{nos_semi}
B. Garc\'{i}a-Archilla and J. Novo.
\newblock POD-ROM methods: from a finite
set of snapshots to continuous-in-time approximations.
arXiv:2403.06967 [math.NA], {\tt  	
https://doi.org/10.48550/arXiv.
2403.06967}

\bibitem{iliescu_et_al_q}
{T.~Iliescu and Z.~Wang}, { Are the snapshot difference quotients needed
  in the proper orthogonal decomposition?} {\em SIAM J. Sci.\ Comput.\/}, 36 (2014),
  A1221--A1250, {\tt https://doi.org/10.1137/130925141}.
  
  \bibitem{iliescu_wang}
{T.~Iliescu and Z.~Wang}, {Variational multiscale proper orthogonal decomposition: Navier-Stokes equations}
{\em Numer.\ Meth.\ PDEs\/}, 30 (2) (2014), 641--663.
  
  
  \bibitem{nos_pos_supg}
{V.~John, B.~Moreau, and J.~Novo}, Error analysis of a
  {SUPG}-stabilized {POD}-{ROM} method for convection-diffusion-reaction
  equations, {\em Comput.\ Math.\ Appl.\/}, 122 (2022), 48--60,
  {\tt https://doi.org/10.1016/j.camwa.2022.07.017},
  
\bibitem{kean_sch}
{K.~Kean and M.~Schneier}, {Error analysis of supremizer pressure
  recovery for {POD} based reduced-order models of the time-dependent
  {N}avier-{S}tokes equations}, SIAM J. Numer. Anal., 58 (2020),
  pp.~2235--2264, {\tt https://doi.org/10.1137/19M128702X}.


\bibitem{samu_et_al}
B.~Koc, S.~Rubino, M.~Schneier, J.~Singler, and T.~Iliescu.
\newblock On optimal pointwise in time error bounds and difference quotients
  for the proper orthogonal decomposition.
\newblock {\em SIAM J. Numer. Anal.}, 59(4) (2021), 2163--2196.

\bibitem{koc_chacon-rubino} 
B.~Koc, T. Chac\'on Rebollo, S.~Rubino, Uniform bounds with difference quotients for proper orthogonal decomposition reduced order models of the Burger's equation,
{\em J. Sci.\ Comput.\/}, 95(2) (2023), 43.

\bibitem{Ku-Vol}
K.~Kunisch and S.~Volkwein.
\newblock Galerkin proper orthogonal decomposition methods for parabolic
  problems.
\newblock {\em Numer. Math.}, 90(1) (2001), 117--148.

\bibitem{Kunisch-Volkwein2}
K.~Kunisch and S.~Volkwein.
\newblock Galerkin proper orthogonal decomposition methods for a general equation in fluid dynamics
\newblock {\em SIAM J. Numer.\ Anal.}, 40 (2002), 492--515.


\bibitem{Joan_toros} J. S\'anchez, M. Net and C. Sim\'o,
\newblock Computation of invariant tori by Newton-Krylov methods in large-scale dissipative systems.
\newblock {\em Physica D}, 59(4) (2021), 2163--219.
 {\tt https://doi.org/10.1016/j.physd.2009.10.012}.

\bibitem{singler}
{\sc J.~R. Singler}, New {POD} error expressions, error bounds, and
  asymptotic results for reduced order models of parabolic {PDE}s, {\em SIAM J.
  Numer. Anal.\/}, 239 (2010), 123--133,
 {\tt https://doi.org/10.1137/120886947}.

\end{thebibliography}
\end{document}